\documentclass[11pt]{article}

\usepackage[english]{babel} 

\usepackage[T1]{fontenc}
\usepackage[dvips]{color}
\usepackage{amsmath,amssymb,amsthm,amsfonts,amscd,authblk,dsfont,cancel,diagbox,
enumerate,epsfig,eucal,fancyhdr,latexsym,lmodern,mathrsfs,mathtools,multirow,
musicography,skull,simplewick,subcaption,tikz-cd}
\usepackage[normalem]{ulem} 

\usepackage[margin=2.41cm]{geometry}
\usepackage{titlesec}
\titleformat{\section}{\large\bfseries\filcenter}{\thesection}{1em}{}
\titleformat{\subsection}{\bfseries}{\thesubsection}{1em}{}
\usepackage[multiuser]{fixme}
\FXRegisterAuthor{nd}{annd}{Nicolo} 
\FXRegisterAuthor{ng}{anng}{Nicolas} 
\FXRegisterAuthor{sm}{ansm}{Simone} 

\usepackage[autostyle, english = american]{csquotes}
\MakeOuterQuote{"} 

\numberwithin{equation}{section}
\newtheorem{theorem}{Theorem}[section]
\newtheorem{lemma}[theorem]{Lemma}
\newtheorem{proposition}[theorem]{Proposition}

\theoremstyle{definition}

\theoremstyle{remark}
\newtheorem{remark}[theorem]{Remark}
\newtheorem{remarks}[theorem]{Remarks}

\usepackage{titlesec}
\titleformat{\section}{\large\bfseries\filcenter}{\thesection}{1em}{}
\titleformat{\subsection}{\bfseries}{\thesubsection}{1em}{}

\usepackage{hyperref}
\hypersetup{
pdftitle={},%
pdfauthor={},%
pdfsubject={},%
pdfkeywords={},%
colorlinks=true,%
linkcolor=black,%
urlcolor=black,
linktoc={},
linktocpage={false},%
pageanchor={},
citecolor={black},
}

\usepackage{cite}

\newcommand{\cf}{\textit{cf. }}

\newcommand{\RR}{{\mathbb{R}}}

\newcommand{\n}{{\mathtt{n}}}

\newcommand{\E}{\mathsf{E}}
\newcommand{\F}{\mathsf{F}}

\newcommand{\M}{\mathsf{M}}

\newcommand{\T}{\mathsf{T}}

\newcommand{\V}{\mathsf{V}}

\newcommand{\bM}{{\partial\M}}
\newcommand{\bSigma}{{\partial\Sigma}}

\newcommand{\oS}{\mathsf{S}}

\newcommand{\oB}{\mathsf{B}}

\def \bui#1#2{\mathrel{\mathop{\kern 0pt#1}\limits^{#2}}}

\newcommand{\End}{{\textnormal{End}\,}}
\newcommand{\supp}{{\textnormal{supp\,}}}

\newcommand{\fiber}[2]{\prec  #1\,|\, #2  \succ}

\definecolor{NiColor}{RGB}{77,77,255}

\begin{document}

\begin{center}

\vspace{5mm}

{\Large\bf  ON THE CAUCHY PROBLEM FOR \\[3mm] THE FARADAY TENSOR ON \\[3mm]  GLOBALLY HYPERBOLIC MANIFOLDS \\[5mm]
WITH TIMELIKE BOUNDARY}


\vspace{5mm}

{\bf by}

\vspace{5mm}
 { \bf Nicol\'o Drago$^1$,  Nicolas Ginoux $^2$ and Simone Murro $^3$}\\[1mm]
\noindent {\it $^1$ Dipartimento di Matematica, Universit\`a di Trento \& INFN-TIFPA, 
Italy 
}\\[1mm]
\noindent {\it  $^2$ Universit\'e de Lorraine \& CNRS \& IECL, F-57000 Metz, France
}\\[1mm]
\noindent {\it  $^3$ Dipartimento di Matematica, Universit\`a di Genova \& INdAM \& INFN, Italy 
}\\[4mm]
email: \ {\tt  nicolo.drago@unitn.it; nicolas.ginoux@univ-lorraine.fr; murro@dima.unige.it} \;\; 
\\[10mm]
\end{center}

\begin{abstract}
We study the well-posedness of the Cauchy problem for the Faraday tensor on globally hyperbolic manifolds with timelike boundary.
The existence of Green operators for the operator $\mathrm{d}+\delta$ and a suitable pre-symplectic structure on the space of solutions are discussed.
\end{abstract}


\section{Introduction}
Electromagnetic interactions play a key role in the history of physics since they are related to the first successful example of unification of two 
apparently different fields, the electric and the magnetic one, into a single body, the Faraday tensor.
The latter tensor contains all the physical information both at a classical and at a quantum level.
Indeed, as noted for example in~\cite{Sak}, in all idealized and real experiments of the Aharonov-Bohm kind, the true observable is actually the flux of the magnetic component of the Faraday tensor which is present inside an impenetrable region, typically a cylinder.
It is far from the scope of this paper to discuss the details of this procedure, but it is sufficient to say that, on Minkowski background and in absence of sources, the
result is pretty much satisfactory.
Yet the situation starts to complicate itself as soon as it is assumed that a spacetime $\M$ has a non-trivial geometry.

In this paper we will be interested in the Cauchy problem for Maxwell's equations (for $k$-forms) $\delta F=j$ and $\mathrm{d}F=0$ on a globally hyperbolic manifold $\M$ with timelike boundary \cite{Ake-Flores-Sanchez-18}.
Within this setting, boundary conditions have to be imposed to ensure the well-posedness of the resulting Cauchy problem: For the case at end we will impose the vanishing of the normal component of the Faraday tensor $F$ at the boundary.

The well-posedness of the Cauchy problem allows to introduce advanced/retarded propagators for the operator $D=\mathrm{d}+\delta$.
This opens to the possibility of applying a standard quantization scheme \cite[Chap. 3]{aqft2} which is well-established for the case of the Faraday tensor on globally hyperbolic manifolds without boundaries \cite{Dappiaggi-Lang-2012,Dimock-92}, for $U(1)$-gauge theories \cite{Benini_2016,Benini_Dappiaggi_Hack_Schenkel_2014,Benini_Dappiaggi_Schenkel_2014} and for gauge theories on globally hyperbolic manifolds with timelike boundary \cite{Benini_Dappiaggi_Schenkel_2018,Dappiaggi-Drago-Longhi-2020}.


\paragraph{Statement of the problem and main results.}

Through this paper, $(\M,g)$ denotes a globally hyperbolic manifold  with 
timelike boundary $\partial\M$ as defined in \cite[Definition 
2.14]{Ake-Flores-Sanchez-18}, see also e.g. \cite[Definition 
2.1]{GinouxMurro2022}.
In more details, $(\M,g)$ is a connected, oriented smooth 
Lorentzian $n$-dimensional manifold $\M$ with boundary $\partial\M$ such that
$(\partial\M,g|_{\partial\M})$ is a Lorentzian manifold and there exists a smooth Cauchy temporal function $t:\M\to\RR$ such that
\begin{align*}
	\M=\RR\times\Sigma
	\qquad
	g=-\beta^2 d t^2+h_t\,,
\end{align*}
where $\beta:\RR \times \Sigma \to \RR$ is a smooth positive function, $h_t$ 
is a Riemannian metric on each slice
$\Sigma_t:=\{t\} \times \Sigma$ varying smoothly with $t$, and these slices are 
spacelike Cauchy hypersurfaces with boundary 
$\bSigma_t:=\{t\}\times\partial\Sigma$, namely achronal sets intersected 
exactly 
once by every inextensible timelike curve. 

The (sourceless) Maxwell's equations for the Faraday tensor $F\in\Omega^k(\M)$ are given by
satisfying 
\begin{equation*}
\mathrm{d}F=0 \qquad \text{and}\qquad \mathrm{d}\ast_g F=0\,.
\end{equation*}
 Clearly, if the boundary of $\bM$ is not empty, then the uniqueness of a 
solution to the Cauchy problem for $F$ can be expected only if a boundary 
condition is imposed. 
To this end,  we shall consider the boundary condition 
\begin{equation*}
\n\lrcorner F=0 
\end{equation*}
where the vector field ${\sf n}$ is the outward-pointing unit normal vector 
field along $\partial\M$. 
If we fix $\nu\in\Gamma(T^*\M_{|_{\partial\M}})$ to be a $1$-form such that 
$$\ker\nu_p=T_p\partial \M \,, \qquad\nu_p(\n_p)>0 \,,  \qquad 
\text{and}  \qquad \mathcal{L}_{\partial_t}\nu=0\,,$$
for all $p\in\partial\M$, then there exists a positive smooth function $c_t$ on 
$\partial\M$ such that 
$$\n_p=c_t\nu^{\sharp_t}$$ 
for all $p\in\partial\M$,
where $\sharp_t\colon T^*\Sigma\to T\Sigma$ denotes the musical isomorphism 
associated with $h_t$.
For later convenience we set
\begin{align*}
	\Omega_{c,\n}^k(\M)
&:=\{F\in\Omega_c^k(\M)\,|\,\n\lrcorner F=0\}\,,\\
	\Omega^\bullet_{c,\n,\delta}(\M)
	&:=\{\alpha\in\Omega_{c,\n}^\bullet(\M)\,|\,\delta \alpha=0\}\,,
	\\
	\Omega^\bullet_{c,\n,\mathrm{d}}(\M)
	&:=\{\alpha\in\Omega_{c,\n}^\bullet(\M)\,|\,\mathrm{d}\alpha=0\}\,.
\end{align*}

\noindent Within this setting, the main result of the paper is the following:
\begin{theorem}\label{Thm: well-posedness for Faraday's tensor initial-value problem with boundary conditions}
	Let $(M,g)$ be a globally hyperbolic manifold with timelike boundary and let be $j\in\Omega_{c,\n,\delta}^{k-1}(\M)$, $\zeta\in\Omega_{c,\mathrm{d}}^{k+1}(\M)$ and $F_0\in\Omega^k_c(\M)$  such that
	\begin{align*}
		(\operatorname{supp}(\zeta)\cup\operatorname{supp}(j))\cap\Sigma_0=\varnothing\,, \quad
		\operatorname{supp}(F_0)\cap\partial\M=\varnothing\,,\quad
		\mathrm{d}_{\Sigma_0} \iota_{\Sigma_0}^* F_0=0\,,\quad
		\mathrm{d}_{\Sigma_0}\iota_{\Sigma_0}^*\ast_g F_0=0\,.
	\end{align*}
	Then the Cauchy problem for the Faraday tensor
	\begin{subequations}\label{Eq: Faraday's tensor initial-value problem with boundary conditions}
		\begin{align}
			\label{Eq: F-homogeneous equation}
			\mathrm{d}F&=\zeta\\
			\label{Eq: F-inhomogeneous equation}
			\delta F&=j\\
			\label{Eq: F normal-boundary condition}
			\n\lrcorner F&=0\\
			\label{Eq: F-initial data}
			F|_{\Sigma_{0}}&=F_0
		\end{align}
	\end{subequations}
	has a unique solution $F\in\Omega_{sc,\n}^k(\M)$.
	Moreover,
	\begin{align}\label{Eq: finite speed of propagation}
		\operatorname{supp}(F)
		\subseteq J\left[\operatorname{supp}(F_0)\cup\operatorname{supp}(j)\cup\operatorname{supp}(\zeta)\right]\,,
	\end{align}
	where $J(A)$ denotes the causal development of $A$.
\end{theorem}

\begin{remarks}\noindent
\begin{enumerate}
	\item
	It is worth pointing out that Theorem \ref{Thm: well-posedness for Faraday's tensor initial-value problem with boundary conditions} proves that any closed compactly supported form $\zeta\in\Omega_{c,\mathrm{d}}^{k+1}(\M)$ is necessarily exact, $\zeta=\mathrm{d}F$, for a spacelike form $F\in\Omega_{sc}^k(\M)$.
	(A similar argument applies for the coexactness of $j$ in Equation \eqref{Eq: F-inhomogeneous equation}.)
	Actually, the inclusion $\Omega_{c,\mathrm{d}}^{k+1}(\M)\subset\mathrm{d}\Omega^k(\M)$ can be proved by cohomological arguments\footnote{We are grateful to M. Benini for this observation.} and is based on the fact that $\M$ is homeomorphic to $\RR\times\Sigma$.
	Indeed, let $f\in C_c^\infty(\mathbb{R})$ be such that $\int_{\mathbb{R}}f(t)\mathrm{d}t=1$ and consider the following maps between chain complexes
	\begin{align*}
		\Omega_c^{\bullet-1}(\Sigma)
		\xrightarrow{f\mathrm{d}t\wedge\cdot }  \Omega_c^{\bullet}(\M)
		\xrightarrow{\operatorname{Id}} \Omega^{\bullet}(\M)
		\xrightarrow{\iota_\Sigma^*} \Omega^{\bullet}(\Sigma)\,,
	\end{align*}
	where the $\operatorname{Id}$ is the identity map while $\iota_\Sigma^*$ is the pull-back to $\Sigma$.
	All these maps induces (de Rham) cohomology maps ---denoted by $[f\mathrm{d}t\wedge\cdot]$, $[\operatorname{Id}]$, $[\iota_\Sigma^*]$--- and by \cite[Prop. 4.7]{Bott_1982} we have $H_c^{\bullet-1}(\Sigma)\simeq H^\bullet_c(\M)$ while \cite[Prop. 4.1]{Bott_1982} proves that $ H^\bullet(\M) \simeq H^\bullet(\Sigma)$.
	Let now $[\omega]_c\in H^\bullet_c(\M)$.
	Since $H^\bullet_c(\M)\simeq H^{\bullet-1}_c(\M)$ there exists $[\alpha]_c\in H^{\bullet-1}_c(\Sigma)$ such that $[\omega]_c=[f\mathrm{d}t\wedge\alpha]_c$.
	Considering the equivalence class $[\operatorname{Id}][f\mathrm{d}t\wedge\alpha]_c=[f\mathrm{d}t\wedge\alpha]\in H^\bullet(\M)$ and the isomorphism $H^\bullet(\M)\simeq H^\bullet(\Sigma)$ we then find
	\begin{align*}
		[f\mathrm{d}t\wedge\alpha]
		=[\iota_\Sigma^*]^{-1}[\iota_\Sigma^*][f\mathrm{d}t\wedge\alpha]
		=[\iota_\Sigma^*]^{-1}[\iota_\Sigma^*(f\mathrm{d}t\wedge\alpha)]
		=[0]\,,
	\end{align*}
	where in the last line we used $\pi_\Sigma \circ f\mathrm{d}t\wedge\cdot=0$.
	This proves that $[\operatorname{Id}]$ is the zero map, hence the claim.
		
	\item Our analysis extends straightforwardly to the Cauchy problem for a Faraday tensor coupled with the boundary condition $ \n\lrcorner \ast_g F_0=0\,.$
		
	\item
	Theorem \ref{Thm: well-posedness for Faraday's tensor initial-value problem with boundary conditions} can be generalized by dropping the assumption $\operatorname{supp}(F_0)\cap\partial\M=\varnothing$ and $(\operatorname{supp}(\zeta)\cup\operatorname{supp}(j))\cap\Sigma_0=\varnothing$.
	This requires introducing suitable "compatibility conditions" between $F_0$ and $j$ as described in \cite{GinouxMurro2022}.
	We will refrain from discussing this case as the hypotheses of Theorem \ref{Thm: well-posedness for Faraday's tensor initial-value problem with boundary conditions} are sufficient for the application we have in mind, \textit{cf.} Proposition \ref{Prop: advanced, retarded propagators}.
		
	\item
	The boundary condition \eqref{Eq: F normal-boundary condition} can be derived with the following variational argument ---\textit{cf.} \cite[Rmk. 27]{Dappiaggi-Drago-Longhi-2020}.
	In this setting one introduces the formal action
		\begin{align*}
			I(A)=\frac{1}{2}(\mathrm{d}A,\mathrm{d}A)_\M
			=\frac{1}{2}\int_\M \mathrm{d}A\wedge\ast_g \mathrm{d}A\,,
		\end{align*}
		where the convergence of the integral is not discussed.
		The homogeneous Maxwell's equations $\delta\mathrm{d}A=0$ are recovered by requiring $A$ to be a critical point of the formal action $I$, namely
		\begin{align*}
			\frac{\mathrm{d}}{\mathrm{d}\varepsilon}I(A+\varepsilon\alpha)\bigg|_{\varepsilon=0}=0
			\qquad\forall\alpha\in\Omega_c^{k-1}(\M)\,,
		\end{align*}
		where $\alpha\in\Omega^{k-1}_c(\M)$ is an arbitrarily chosen compactly supported smooth $(k-1)$-form.
		Notably, although $I(A)$ may be ill-defined, the derivative $\dfrac{\mathrm{d}}{\mathrm{d}\varepsilon}I(A+\varepsilon\alpha)\bigg|_{\varepsilon=0}$ is always well-defined and it can be written as
		\begin{align*}
			\frac{\mathrm{d}}{\mathrm{d}\varepsilon}I(A+\varepsilon\alpha)\bigg|_{\varepsilon=0}
			=(\mathrm{d}A,\mathrm{d}\alpha)_\M
			=(\delta\mathrm{d}A,\alpha)_\M
			+(\n\lrcorner\mathrm{d}A,\iota_{\partial \M}^*\alpha)_{\partial \M}\,,
		\end{align*}
		where $(\cdot\,,\cdot)_{\partial \M}$ is the canonical pairing between forms on $\partial \M$.
		Because $\alpha$ can be chosen arbitrarily, this leads to $\delta\mathrm{d}A=0$ and $\n\lrcorner \mathrm{d}A=0$.
		\item
		The 
		well-posedness of the Cauchy problem will guarantee the existence 
		of Green operators (\cf Proposition~\ref{Prop: advanced, retarded propagators}) which play a pivotal role 
		in the algebraic approach to 
		linear quantum field theory, see e.g.~\cite{gerard, aqft2, wave} for textbooks and
		\cite{BaerGinoux2011,FK,aqft2} for recent reviews.
	\end{enumerate}
\end{remarks}

\paragraph{Plan of the proof.}
As a preliminary, in Section~\ref{sec:2} we will decompose the Faraday tensor $F\in\Omega^k(\M)$ into its electric and magnetic components $F_E,F_B$, \textit{cf.} Equation~\eqref{Eq: electric and magnetic 2-form}.
The equations of motion \eqref{Eq: F-homogeneous equation}-\eqref{Eq: F-inhomogeneous equation} are then written in terms of $F_E,F_B$ leading to the standard formulation of Maxwell's equations in terms of electric and magnetic "fields".
Within this setting the system made by \eqref{Eq: F-homogeneous equation}-\eqref{Eq: F-inhomogeneous equation} decouples in a system of 2 dynamical equations, which determine $F_E,F_B$ once initial data and boundary conditions are provided, and 2 constraint equations, which must be fulfilled along the motion and in particular by the initial data.
Similarly, the initial condition \eqref{Eq: F-initial data} leads to initial conditions for $F_E,F_B$; moreover, the same applies for the boundary condition \eqref{Eq: F normal-boundary condition} which leads to 2 boundary conditions for $F_E$ and $F_B$.
As we will see more in details, the boundary conditions we obtain are somehow redundant: The first one can be used to determine $F_E,F_B$ uniquely ---together with the initial data and the dynamical equations of motion--- whereas the latter plays the role of a constraint.
Summing up, the initial-value problem with boundary conditions \eqref{Eq: Faraday's tensor initial-value problem with boundary conditions} for $F$ will be turned into an initial-value problem with boundary conditions and constraints for $F_E,F_B$.

In Section~\ref{sec:3}, we will solve the initial-boundary value problem for $F_E,F_B$ relying on the results of \cite{GinouxMurro2022}.
Henceforth, in Section~\ref{sec:4} we will prove that the constraints are fulfilled once they are fulfilled by the initial data.
We conclude our paper with Section~\ref{sec:5}, devoted to prove the existence of Green operators for the Faraday tensor.
This leads to a pre-symplectic form on the space of solutions to the Cauchy problem for the Faraday tensor: To this avail, however, one has to consider Faraday of all degrees in a unified non-trivial fashion.

\paragraph{Acknowledgment.}
We are grateful to Marco Benini for helpful discussions related to the topic of this paper.
This work was produced within the activities of the INdAM-GNFM.
N.D. acknowledges the support of the GNFM-INdAM Progetto Giovani \textit{Non-linear sigma models and the Lorentzian Wetterich equation}.
S.M. acknowledges the support of the INFN-sezione di Genova as well as the support of the GNFM-INdAM Progetto Giovani \textit{Feynman propagator for Dirac fields: a microlocal analytic approach.}

\section{Reformulation of the Cauchy problem}\label{sec:2}

Let $\pi_2\colon \M\to\Sigma$ be the projection on the second factor in the 
Cartesian product $\M=\mathbb{R}\times\Sigma$ and let 
$\V^\bullet:=\pi_2^*(\Lambda^\bullet T^*\Sigma)\to \M$ be the pull-back over $\M$ of the exterior bundle of $\Sigma$.
The \textbf{electric and the magnetic components of a given $F\in\Omega^k(\M)$} are the forms $F_B\in\Gamma(\V^k)$ and $F_E\in\Gamma(\V^{n-k})$ defined by
\begin{align}\label{Eq: electric and magnetic 2-form}
	F=\mathrm{d}t\wedge\ast_{h_t} F_E+F_B\,,
\end{align}
where $\ast_{h_t}$ denotes the Hodge dual with respect to the metric $h_t$.
More explicitly we have
\begin{align*}
\ast_{h_t}F_E:=\partial_t\lrcorner F\,,
\qquad
F_B=F-\mathrm{d}t\wedge\ast_{h_t}F_E\,,
\end{align*}
where $\partial_t\lrcorner$ denotes the interior product with 
$\partial_t$.
Clearly $F_E,F_B$ determines $F$ uniquely and viceversa.

\begin{remark}\label{Rmk: Hodge dual relations}
	For later convenience we shall recollect here some useful identities concerning the differential, codifferential, Hodge operators, pull-backs and interior products.
	Let $(M,g)$ be an $m$-dimensional pseudo-Riemannian manifold with possibly non-empty boundary $\partial M$; in most applications below, $M^m$ will be either the spacetime $\M$, its boundary $\bM$ together with its induced Lorentzian metric or the Cauchy hypersurface $\Sigma$ with Riemannian metric $h_t$.
	We denote by $\sigma_M$ the index of $M$.
	The orientation of $\M$ will be chosen such that, for any oriented pointwise basis $(e_1^*,\ldots,e_{n-1}^*)$ of $T^*\Sigma$, the $n$-tuple $(\mathrm{d}t,e_1^*,\ldots,e_{n-1}^*)$ is an oriented basis of $T^*\M$.
	We denote by $\mathrm{d}_\bullet\colon\Omega^\bullet(M)\to\Omega^{\bullet+1}(M)$ the differential on $M$ while $\ast_\bullet\colon\Omega^\bullet(M)\to\Omega^{m-\bullet}(M)$ denotes the Hodge dual of $(M,g)$.
	To emphasize the difference between operators on $\M$ and on $\bM$, the differential and Hodge dual of $\partial M$ will be denoted by $\mathrm{d}_\bullet^{\partial M}$ and $\ast_\bullet^{\partial M}$ respectively ---we will suppress the superscript when the latter is clear from the context.
	We then have
	\begin{align*}
		\ast_{m-k}\ast_k&=(-1)^{k(m-k)+\sigma_M}
		\quad
		(\Rightarrow
		\ast_k^{-1}=(-1)^{k(m-k)+\sigma_M}\ast_{m-k})
		\\
		\delta_k&=\mathrm{d}_k^*=(-1)^k(\ast_{k-1})^{-1}\mathrm{d}_{m-k}\ast_k
		\\
		\ast_{k-1}\delta_k&=(-1)^k\mathrm{d}_{m-k}\ast_k
		\qquad
		\delta_{m-k}\ast_k=(-1)^{k+1}\ast_{k+1}\mathrm{d}_k
		\\
		\ast^{\partial M}_{k-1}\n\lrcorner
		&=\iota_{\partial M}^*\ast_k
		\\
		\n\lrcorner\ast_{m-k}
		&=(-1)^{m-k+\sigma_M+\sigma_{\partial M}}\ast^{\partial M}_{m-k}\iota_{\partial M}^*
		\\
		X^\flat\wedge\ast_g\omega&=(-1)^{k+1}\ast_g(X\lrcorner\,\omega)\\
		\ast_g(X^\flat\wedge\omega)&=(-1)^kX\lrcorner\ast_g\omega
		\\
		\delta^{\partial M}_{m-k-1}\n\lrcorner|_{\Omega^{m-k}(\M)}
		&=-\n\lrcorner\delta_{m-k}\,,
	\end{align*}
for all $\omega\in\Lambda^kT^*M$ and $X\in TM$.
Moreover, defining the pointwise nondegenerate inner product $\langle\cdot\,,\cdot\rangle$ on $\Lambda^kT^*M$ via
\[\langle\omega,\omega'\rangle:=(-1)^{\sigma_M}\cdot\ast_g(\omega\wedge\ast_g\omega'),\]
we have, for all $X\in TM$, $\omega\in\Lambda^kT^*M$ and $\omega'\in\Lambda^{k+1}T^*M$,
\begin{eqnarray*}
\langle X^\flat\wedge\omega,\omega'\rangle&=&(-1)^{\sigma_M}\cdot\ast_g(X^\flat\wedge\omega\wedge\ast_g\omega')\\
&=&(-1)^{\sigma_M}\cdot(-1)^k\cdot\ast_g(\omega\wedge X^\flat\wedge\ast_g\omega')\\
&=&(-1)^{\sigma_M}\cdot(-1)^k\cdot(-1)^k\ast_g(\omega\wedge\ast_g(X\lrcorner\,\omega'))\\
&=&\langle\omega,X\lrcorner\,\omega'\rangle.
\end{eqnarray*}
Moreover, for all $\omega\in\Lambda^kT^*M$ and $\omega'\in\Lambda^{m-k}T^*M$,
\begin{eqnarray*}
\langle\ast_g\omega,\omega'\rangle&=&(-1)^{k(m-k)+\sigma_M}\cdot\langle\ast_g\omega,\ast_g^2\omega'\rangle\\
&=&(-1)^{k(m-k)+2\sigma_M}\cdot\langle\omega,\ast_g\omega'\rangle\\
&=&(-1)^{k(m-k)}\cdot\langle\omega,\ast_g\omega'\rangle.
\end{eqnarray*}

\end{remark}

The next lemma converts equations \eqref{Eq: F-homogeneous equation}-\eqref{Eq: F-inhomogeneous equation} into dynamical and constraint equations for $F_E,F_B$.

\begin{lemma}\label{Lem: Maxwell's equation for electric and magnetic component}
	A $k$-form $F\in\Omega^k(\M)$ solves \eqref{Eq: F-homogeneous equation}-\eqref{Eq: F-inhomogeneous equation} if and only if its electric 
	and magnetic components $F_E,F_B$ solve
\begin{subequations}\label{Eq: Maxwell's equations for electric and magnetic components}
	\begin{align}
		\label{Eq: E dynamical equation}
		\beta^{-1}\mathcal{L}_{\partial_t}(\beta^{-1}F_E)+(-1)^{(n-k+1)(k+1)+1}\beta^{-1}\mathrm{d}_\Sigma(\ast_{h_t}\beta F_B)&=(-1)^{(n-k)(k+1)}\ast_{h_t}j_B\,,
		\\
		\label{Eq: B dynamical equation}
		\mathcal{L}_{\partial_t}F_B
		-\mathrm{d}_\Sigma\ast_{h_t}F_E&=\ast_{h_t}\zeta_E\,,
		\\
		\label{Eq: E constrained equation}
		\mathrm{d}_\Sigma(\beta^{-1}F_E)&=(-1)^{n-k}\beta^{-1} j_E\,,
		\\
		\label{Eq: B constrainted equation}
		\mathrm{d}_\Sigma F_B&=\zeta_B\,,
	\end{align}
\end{subequations}
where $\mathrm{d}_\Sigma$ denotes the differential on $\Sigma$, while $j_E\in\Gamma(\V^{n+1-k})$ and $j_B\in\Gamma(\V^{k-1})$ are the electric and magnetic components of $j\in\Omega^{k-1}(\M)$.
\end{lemma}
\begin{proof}
	We recall that the differential $\mathrm{d}$ on $\M$ and the differential 
	$\mathrm{d}_\Sigma$ on $\Sigma$ are related by
	\begin{align*}
		\mathrm{d}\omega=\mathrm{d}t\wedge\partial_t\lrcorner\mathrm{d}\omega
		+\mathrm{d}_\Sigma\iota_\Sigma^*\omega\,,
	\end{align*}
	for all $\omega\in\Omega^k(\M)$.
	By direct inspection we have
	\begin{align*}
		\zeta=\mathrm{d}F&=
		-\mathrm{d}t\wedge\mathrm{d}_\Sigma\ast_{h_t}F_E
		+\mathrm{d}t\wedge\partial_t\lrcorner\mathrm{d}F_B
		+\mathrm{d}_\Sigma F_B
		&\textrm{Eq. } \eqref{Eq: electric and magnetic 2-form}
		\\&=\mathrm{d}t\wedge\left[\mathcal{L}_{\partial_t}F_B
		-\mathrm{d}_\Sigma\ast_{h_t}F_E\right]
		+\mathrm{d}_\Sigma F_B
		&\partial_t\lrcorner F_B=0\,,
	\end{align*}
	which leads to Equations \eqref{Eq: B dynamical equation} and \eqref{Eq: B 
	constrainted equation} once we consider the decomposition $\zeta=\mathrm{d}t\wedge\ast_{h_t}\zeta_E+\zeta_B$.

	For what concerns Equations \eqref{Eq: E dynamical equation} and \eqref{Eq: E constrained equation} we consider the Hodge dual of Equation \eqref{Eq: F-inhomogeneous equation}:
	\begin{align*}
		\ast_g j
		=\ast_g\delta F
		=(-1)^k\mathrm{d}\ast_g F\,.
	\end{align*}
	Moreover, for all $\omega\in\Gamma(\V^k)$ we have $\beta dt\wedge\ast_{h_t}\omega=(-1)^k\ast_g\omega$, which implies
\begin{align*}
\ast_g F&=\beta^{-1}\ast_g(\beta\mathrm{d}t\wedge\ast_{h_t} F_E)+\ast_g F_B\\
&=(-1)^{n-k}\beta^{-1}\ast_g^2F_E+\ast_g F_B\\
&=(-1)^{(n-k)(k+1)+\sigma_{\M}}\beta^{-1}F_E+(-1)^k\beta\mathrm{d}t\wedge\ast_{h_t}F_B\\
&=(-1)^{(n-k)(k+1)+1}\beta^{-1}F_E+(-1)^k\beta\mathrm{d}t\wedge\ast_{h_t}F_B
\end{align*}
and similarly
\begin{align*}
	\ast_g j&=\ast_g(\mathrm{d}t\wedge\ast_{h_t}j_E+j_B)\\
	&=\beta^{-1}\ast_g(\beta\mathrm{d}t\wedge\ast_{h_t} j_E)+\ast_g j_B\\
	&=(-1)^{n-k+1}\beta^{-1}\ast_g^2 j_E+(-1)^{k-1}\beta\mathrm{d}t\wedge\ast_{h_t}j_B\\
	&=(-1)^{n-k+1+(n-k+1)(k-1)+\sigma_{\M}}\beta^{-1}j_E+(-1)^{k-1}\beta\mathrm{d}t\wedge\ast_{h_t}j_B\\
	&=(-1)^{k(n-k+1)+1}\beta^{-1}j_E+(-1)^{k-1}\beta\mathrm{d}t\wedge\ast_{h_t}j_B\,.
\end{align*}
Therefore,
\begin{eqnarray*}
	\mathrm{d}\ast_g F&=&(-1)^{(n-k)(k+1)+1}\mathrm{d}(\beta^{-1}F_E)+(-1)^k\mathrm{d}(\mathrm{d}t\wedge\ast_{h_t}\beta F_B)\\
	&=&(-1)^{(n-k)(k+1)+1}\mathrm{d}t\wedge\partial_t\lrcorner\,\mathrm{d}(\beta^{-1}F_E)+(-1)^{(n-k)(k+1)+1}\mathrm{d}_\Sigma(\beta^{-1}F_E)\\
	&&+(-1)^{k+1}\mathrm{d}t\wedge \mathrm{d}_\Sigma(\ast_{h_t}\beta F_B)\\
	&=&\mathrm{d}t\wedge\left((-1)^{(n-k)(k+1)+1}\mathcal{L}_{\partial_t}(\beta^{-1}F_E)+(-1)^{k+1}\mathrm{d}_\Sigma(\ast_{h_t}\beta F_B)\right)\\
	&&+(-1)^{(n-k)(k+1)+1}\mathrm{d}_\Sigma(\beta^{-1}F_E).
\end{eqnarray*}
It can be deduced that $\mathrm{d}\ast_g F=(-1)^k\ast_g j$ if and only if
\[\left\{\begin{array}{rll}(-1)^{(n-k)(k+1)+1}\mathcal{L}_{\partial_t}(\beta^{-1}F_E)+(-1)^{k+1}\mathrm{d}_\Sigma(\ast_{h_t}\beta F_B)&=&-\beta\ast_{h_t}j_B\\(-1)^{(n-k)(k+1)+1}\mathrm{d}_\Sigma(\beta^{-1}F_E)&=&(-1)^{k(n-k)+1}\beta^{-1}j_E\end{array}\right.,\]
that is
\begin{align*}
	\left\{
	\begin{array}{rll}
		\beta^{-1}\mathcal{L}_{\partial_t}(\beta^{-1}F_E)+(-1)^{(n-k+1)(k+1)+1}\beta^{-1}\mathrm{d}_\Sigma(\ast_{h_t}\beta F_B)
		&=&(-1)^{(n-k)(k+1)}\ast_{h_t}j_B
		\\
		\mathrm{d}_\Sigma(\beta^{-1}F_E)
		&=&(-1)^{n-k}\beta^{-1} j_E
	\end{array}\right..
\end{align*}
This leads to Equations \eqref{Eq: E dynamical equation} and \eqref{Eq: E constrained equation}.
\end{proof}

\begin{remark}
	The constraint $\delta j=0$ on the current $j\in\Omega^{k-1}_{c,\n,\delta}(\M)$ assumed in Theorem \ref{Thm: well-posedness for Faraday's tensor initial-value problem with boundary conditions} reduces to the standard continuity equation in terms of $j_E,j_B$:
	\begin{align}\label{Eq: j continuity equation}
		\mathcal{L}_{\partial_t}[\beta^{-1}j_E]
		+(-1)^{k(n-k)+1}\mathrm{d}_\Sigma[\beta\ast_{h_t} j_B]=0\,,
		\qquad
		\mathrm{d}_\Sigma[\beta^{-1}j_E]=0\,.
	\end{align}
	Similarly $\zeta\in\Omega_{c,\mathrm{d}}^{k+1}(\M)$ has to be closed, therefore,
	\begin{align}\label{Eq: zeta continuity equation}
		\mathcal{L}_{\partial_t}\zeta_B-\mathrm{d}_\Sigma\ast_{h_t}\zeta_E=0\,,
		\qquad
		\mathrm{d}_\Sigma\zeta_B=0\,.
	\end{align}
\end{remark}

Thus, Equations \eqref{Eq: F-homogeneous equation}-\eqref{Eq: F-inhomogeneous equation} can be recast into Equations \eqref{Eq: Maxwell's equations for electric and magnetic components}.
Notice that the latter consists of two dynamical equations \eqref{Eq: E dynamical equation}-\eqref{Eq: B dynamical equation} and two constraint equations \eqref{Eq: E constrained equation}-\eqref{Eq: B constrainted equation}.
In the next section we will prove that Equations \eqref{Eq: E dynamical equation}-\eqref{Eq: B dynamical equation} define a symmetric hyperbolic system \cite[Def. 2.4-2.5]{GinouxMurro2022}.
Before that, we observe that the boundary condition \eqref{Eq: F normal-boundary condition} can be equivalently written in terms of the electric and magnetic components $F_E,F_B$ as
\begin{align}
	\label{Eq: B boundary condition}
	\n\lrcorner\ast_{h_t}F_E=0
	\qquad(\Leftrightarrow\iota_{\partial\Sigma_t}^*F_E=0)
	\\
	\label{Eq: E boundary condition}
	\n\lrcorner F_B=0
	\qquad(\Leftrightarrow\iota_{\partial\Sigma_t}^*\ast_{h_t}F_B=0)\,.
\end{align}
As we will see, in order to apply the results of \cite{GinouxMurro2022} only one among \eqref{Eq: B boundary condition}-\eqref{Eq: E boundary condition} is needed ---in the following we will choose \eqref{Eq: B boundary condition}.
The remaining boundary condition is redundant, in fact, it plays the role of an additional constrained boundary condition.

\section{Maxwell's equations as a constrained symmetric hyperbolic system}\label{sec:3}

We now recast Equations \eqref{Eq: E dynamical equation}-\eqref{Eq: B dynamical equation} into a symmetric hyperbolic system.
Following \cite[Def. 2.4-2.5]{GinouxMurro2022} we recall that a differential operator $\oS\colon\Gamma(\E)\to\Gamma(\E)$ on a Riemannian vector bundle $\E\to\M$, is called \textbf{symmetric hyperbolic system} over $\M$ if
\begin{enumerate}
	\item[(S)]
	The principal symbol $\sigma_\oS (\xi) \colon \E_p \to \E_p$ is pointwise self-adjoint resp. symmetric with respect to $\fiber{\cdot}{\cdot}_p$ for every $\xi\in \T^*_p\M$ and for every $p \in \M$ ---here $\fiber{\cdot}{\cdot}_p$ denotes the Riemannian resp. symmetric fiber pairing at $\E_p$;
	
	\item[(H)]\label{conditionH}
	For every future-directed timelike covector $\tau\in\T_p^*\M$, the bilinear form $\fiber{\sigma_\oS (\tau) \cdot}{\cdot}_p$ is positive definite on $\E_p$ for every $p\in\M$.
\end{enumerate}
A symmetric hyperbolic system $\oS$ is said \textbf{of constant characteristic} if $\dim\ker \sigma_\oS(\n^\flat)$ is constant, where $\sigma_{\oS}(\n^\flat)\in\End(T^*\M|_{\partial\M})$.
In particular, if $\sigma_\oS(\n^\flat)$ has maximal rank at each point of $\bM$ we say that $\oS$ is \textbf{nowhere characteristic}.

Concerning boundary conditions for a symmetric hyperbolic system $\oS$ with constant characteristic we quote from \cite[Definition 2.13]{GinouxMurro2022}.
A smooth subbundle $\oB$ of $\E_{|_\bM}$ is called a \textbf{self-adjoint admissible boundary condition} for $\oS$ if
\begin{enumerate}[(i)]
	\item\label{Item: bc condition - vanishing of principal symbol}
	the quadratic form $\Psi\mapsto\fiber{\sigma_{\oS}(\nu)\Psi}{\Psi}_p$ vanishes on $\oB$ ---here $\nu\in\Omega^1(\M)$ is any form such that $\ker\nu_x=T_x\partial\M$ for all $x\in\partial\M$;
	\item\label{Item: bc condition - rank of principal symbol}
	the rank of $\oB$ is equal to the number of pointwise non-negative eigenvalues of $\sigma_{\oS}(\nu)$ counting multiplicity;
	\item\label{Item: bc condition - self-adjointness}
	the identity $\oB=\oB^\dagger$ holds, where $\oB^\dagger:=[\sigma_{\oS}(\n^\flat)\mathsf{B}]^\perp$ and the symbol $(\cdot)^\perp$ denotes the pointwise orthogonal complement with respect to $\fiber{\cdot}{\cdot}$.
\end{enumerate}

The next Proposition shows that Equations \eqref{Eq: E dynamical equation}-\eqref{Eq: B dynamical equation} can be interpreted as a symmetric hyperbolic system of constant characteristic.
Moreover, the boundary condition \eqref{Eq: B boundary condition} is a self-adjoint boundary condition for that  symmetric hyperbolic system.

\begin{proposition}\label{Prop:symmetric hyperbolic system equivalent to Maxwell}
	Let $\E=\V^{n-k}\oplus \V^k\to \M$ be the vector bundle over $\M$ with the standard positive-definite fiber metric $\fiber{\cdot}{\cdot}$ between forms.
	Actually for $F_B,F_B'\in\Gamma(\V^{k})$ we have
	\begin{align*}
		\fiber{F_B}{F_B'}
		:=\ast_{h_t}[F_B\wedge\ast_{h_t} F_B']
		=-\ast_g[F_B\wedge\ast_g F_B']\,.
	\end{align*}
	Then:
	\begin{enumerate}
		\item\label{Item: symmetric hyperbolic system for Maxwell}
		The first-order differential operator $\oS\colon\Gamma(\E)\to\Gamma(\E)$ defined by 
		\begin{align}\label{def:S}
			\oS\left[\begin{matrix}F_E\\F_B \end{matrix}\right]
			&=\left(\begin{matrix}
				\beta^{-1}\mathcal{L}_{\partial_t}\circ\beta^{-1}
				&(-1)^{(n-k+1)(k+1)+1}\beta^{-1}\mathrm{d}_\Sigma\ast_{h_t}\beta
				\\
				-\mathrm{d}_\Sigma\ast_{h_t}&\mathcal{L}_{\partial_t}
			\end{matrix}\right)\left[\begin{matrix}F_E\\F_B \end{matrix}\right]\,,
		\end{align}
		is a symmetric hyperbolic system of constant characteristic.
		\item\label{Item: admissibile boundary condition for Maxwell}
		The subbundle $\mathsf{B}\subset\E|_{\partial\M}$ defined by
		\begin{align}\label{Eq: subbundle for admissible boundary condition for Maxwell}
			\mathsf{B}:=\{(F_E,F_B)\in\E|_{\partial\M}\,|\,\n\lrcorner F_B=0\}
			:=\{(F_E,F_B)\in\E|_{\partial\M}\,|\,\nu\wedge\ast_{h_t}F_B=0\}\,,
		\end{align}
		defines a self-adjoint admissible boundary condition for $\oS$.
	\end{enumerate}
\end{proposition}
\begin{proof}
	\noindent
	\begin{description}
		\item[\boxed{\ref{Item: symmetric hyperbolic system for Maxwell}}]
		The principal symbol of $\oS$ at $\xi\in T_p^*\M$, $p\in\Sigma_t$, is given by
		\begin{align*}
			\sigma_{\oS}(\xi)=\left(
			\begin{matrix}
				\beta^{-2}\xi(\partial_t)\operatorname{Id}_{\V^{n-k}|_p}
				&(-1)^{(n-k+1)(k+1)+1}\xi_{\Sigma_t}\wedge\ast_{h_t}
				\\
				-\xi_{\Sigma_t}\wedge\ast_{h_t}&\xi(\partial_t)\operatorname{Id}_{\V^{k}|_p}
			\end{matrix}
			\right)\,,
		\end{align*}
		where $\xi_{\Sigma_t}:=\iota_{\Sigma_t}^*\xi$ being 
		$\iota_{\Sigma_t}\colon\Sigma_t\to \M$.
		By direct inspection we have, for all $F_E\in \V^{n-k}_p$, $F_B\in \V^k_p$, and $\xi\in T_p^*\M$,
		\begin{align*}
		\fiber{-\xi_{\Sigma_t}\wedge\ast_{h_t}F_E}{F_B}&=-\fiber{\ast_{h_t}F_E}{\xi_{\Sigma_t}^{\sharp_t}\lrcorner\,F_B}\\
		&=-(-1)^{(n-k)(k-1)}\fiber{F_E}{\ast_{h_t}(\xi_{\Sigma_t}^{\sharp_t}\lrcorner\,F_B)}\\
		&=-(-1)^{(n-k)(k-1)+k+1}\fiber{F_E}{\xi_{\Sigma_t}\wedge\ast_{h_t}F_B}\\
		&=(-1)^{(n-k+1)(k+1)+1}\fiber{F_E}{\xi_{\Sigma_t}\wedge\ast_{h_t}F_B}\,,
		\end{align*}

		which shows $\sigma_{\oS}(\xi)^\dagger=\sigma_{\oS}(\xi)$ and therefore that
		condition~(S) holds.\\
		Next we prove condition~(H).
		Let $\xi=\xi(\partial_t)dt+\xi_{\Sigma_t}\in T_p^*\M$ be any future-directed 
		timelike covector that is, $\|\xi_{\Sigma_t}\|_{h_t}^2<\beta^{-2}\xi(\partial_t)^2$ and 
		$\xi(\partial_t)>0$.
		For any $F_E\in\V^k_p$ and $F_B\in\V^{n-k}_p$ we have
		\begin{align*}
			\fiber{\sigma_{\oS}(\xi)(F_E,F_B)}{(F_E,F_B)}
			&=\beta^{-2}\xi(\partial_t)\fiber{F_E}{F_E}
			+\xi(\partial_t)\fiber{F_B}{F_B}
			\\&-2\fiber{\xi_{\Sigma_t}\wedge\ast_{h_t} F_E}{F_B}
			\\&\geq\beta^{-2}\xi(\partial_t)\fiber{F_E}{F_E}
			+\xi(\partial_t)\fiber{F_B}{F_B}
			\\&-2\|\xi_{\Sigma_t}\|_{h_t}
			\fiber{F_E}{F_E}^{1/2}
			\fiber{F_B}{F_B}^{1/2}
			\\&\geq\beta^{-2}\xi(\partial_t)\fiber{F_E}{F_E}
			+\xi(\partial_t)\fiber{F_B}{F_B}
			\\&-2\beta^{-1}\xi(\partial_t)\fiber{F_E}{F_E}^{1/2}
			\fiber{F_B}{F_B}^{1/2}
			\\&=\xi(\partial_t)\left[\beta^{-1}\fiber{F_E}{F_E}^{1/2}-\fiber{F_B}{F_B}^{1/2}\right]^2
			\geq 0\,.
		\end{align*}
		Moreover, if $\fiber{\sigma_{\oS}(\xi)(F_E,F_B)}{(F_E,F_B)}=0$ then the above inequalities implies
		\begin{align*}
			\|\xi_{\Sigma_t}\|_{h_t}\fiber{F_E}{F_E}\fiber{F_B}{F_B}
			=\beta^{-2}\xi(\partial_t)^2\fiber{F_E}{F_E}\fiber{F_B}{F_B}\,,
		\end{align*}
		which forces $F_E=0$ and $F_B=0$ due to the condition $\|\xi_{\Sigma_t}\|_{h_t}^2<\xi(\partial_t)^2\beta^{-2}$.
		This proves that $\sigma_{\oS}(\xi)$ is positive definite and therefore condition~(H) holds.
		
		Finally, since $\sigma_{\oS}(\nu)$ is given by
		\begin{align*}
			\sigma_{\oS}(\nu)=\left(
			\begin{matrix}
				0&(-1)^{k(n-k)}\nu\wedge\ast_{h_t}
				\\
				-\nu\wedge\ast_{h_t}&0
			\end{matrix}
			\right)\,,
		\end{align*}
		it follows that 
		\begin{eqnarray*}
		\ker\sigma_\oS(\nu)
			&=&\{
			(F_E,F_B)\in\V^{n-k}\oplus\V^k\,|\,
			\n\lrcorner F_E
			=0=\n\lrcorner F_B
			\}
			\\
			&=&\pi_2^*\Lambda^{n-k}T^*\partial\Sigma\oplus
			\pi_2^*\Lambda^kT^*\partial\Sigma\,,
		\end{eqnarray*}
which proves that $\oS$ is of constant characteristic.
		
		\item[\boxed{\ref{Item: admissibile boundary condition for Maxwell}}]
		We now prove that the subbundle $\mathsf{B}$ introduced in Equation \eqref{Eq: subbundle for admissible boundary condition for Maxwell} identifies a future admissible boundary condition for $\oS$.
		By direct inspection we have
		\begin{align*}
			\E_{|_{\partial\M}}=
			\ker\sigma_{\oS}(\nu)
			\oplus\ker[\sigma_{\oS}(\nu)+1]
			\oplus\ker[\sigma_{\oS}(\nu)-1]\,,
		\end{align*}
		where 
		\begin{align*}
			\ker\sigma_{\oS}(\nu)
			&\simeq\pi_2^*\Lambda^{n-k}T^*\partial\Sigma\oplus\pi_2^*\Lambda^kT^*\partial\Sigma\,,
			\\
			\ker[\sigma_{\oS}(\nu)-\varepsilon]
			&=\{(F_E,-\varepsilon\nu\wedge\ast_{h_t}F_E)\in \E_{|\partial\M}\,|\,
			\ast_{h_t}F_E\in\pi_2^*\Lambda^{k-1}T^*\partial\Sigma\}\,,
			\qquad\varepsilon\in\{1,-1\}\,.
		\end{align*}
		Notice $\dim\ker\sigma_{\oS}(\nu)={n-2\choose n-k}+{n-2\choose k}$, moreover, each eigenspace associated to $\varepsilon\in\{\pm1\}$ has pointwise rank ${n-2\choose k-1}$.
		Thus, an admissible boundary condition must have rank ${n-2\choose k}+{n-2\choose k-1}+{n-2\choose n-k}$ because of condition \eqref{Item: bc condition - rank of principal symbol}.
		But this is exactly the case for $\mathsf{B}$, whose dimension is ${n-1\choose k-1}+{n-2\choose k}$ so that condition \eqref{Item: bc condition - rank of principal symbol} is fulfilled.
		Moreover, for all $(F_E,F_B)\in\mathsf{B}$ it holds
		\begin{align*}
			\fiber{\sigma_{\oS}(\nu)(F_E,F_B)}{(F_E,F_B)}
			=-2\fiber{F_E}{\nu\wedge\ast_{h_t} F_B}=0\,.
		\end{align*}
		The latter equality implies condition \eqref{Item: bc condition - vanishing of principal symbol}.
		Finally, since $\oB=\V^{n-k}\oplus\pi_2^*\Lambda^k T^*\partial\Sigma$ and $\sigma_{\oS}(\nu)(\oB)=\{(0,-\nu\wedge\ast_{h_t} F_E)\,|\, F_E\in \V^{n-k}\}$ we have that $\oB^\dagger=\V^{n-k}\oplus\pi_2^*\Lambda^kT^*\partial\Sigma=\oB$ \textit{i.e.} condition \eqref{Item: bc condition - self-adjointness} is fulfilled.
	\end{description}
	This concludes our proof.
\end{proof}


\section{The Cauchy problem for the Faraday tensor}\label{sec:4}
We have finally all the ingredients to prove our main theorem.
\begin{proof}[Proof of Theorem~\ref{Thm: well-posedness for Faraday's tensor initial-value problem with boundary conditions}.]
On account of Lemma \ref{Lem: Maxwell's equation for electric and magnetic component} we may reduce our problem to the initial-value problem
\begin{subequations}\label{Eq: reduced dynamical Maxwell system}
	\begin{align}
		\oS(F_E,F_B)
		&=((-1)^{(n-k)(k+1)}\ast_{h_t} j_B,\ast_{h_t}\zeta_E)\,,
		\\
		(F_E,F_B)|_{\Sigma_0}
		&=(F_{0,E},F_{0,B})
		\\
		(F_E,F_B)|_{\partial\M}
		&\in\mathsf{B}
	\end{align}
\end{subequations}
subjected to the constraint equations
\begin{align}\label{Eq: reduced constraint Maxwell system}
	\mathrm{d}_\Sigma[\beta^{-1}F_E]=(-1)^{n-k}\beta^{-1}j_E\,,
	\qquad
	\mathrm{d}_\Sigma F_B=\zeta_B\,,
	\qquad
	\iota_{\partial\Sigma_t}^*F_E=0\,.
\end{align}
Here $F_{0,E,}, F_{0,B}$ denote the electric and magnetic component of the initial datum $F_0\in\Omega^k(\M)$.
Notice that the assumptions on the initial data $F_0$ implies
\begin{align*}
	(F_{0,E},F_{0,B})\in\mathsf{B}\,,
	\qquad
	\mathrm{d}_\Sigma[\beta^{-1}F_{0,E}]=0\,,
	\qquad
	\mathrm{d}_\Sigma F_{0,B}=0\,,
\end{align*}
Since $\oS$ is symmetric hyperbolic and $\mathsf{B}$ is an admissible self-adjoint boundary condition for $\oS$, we may apply \cite[Thm. 1.2]{GinouxMurro2022}.
Notice that the compatibility conditions mentioned therein ---\textit{cf.} \cite[Eq. (4.3)]{GinouxMurro2022}--- are automatically fulfilled on account of our assumption that $\operatorname{supp}(F_0)\cap\partial\M=\varnothing$ and $(\operatorname{supp}(\zeta)\cup\operatorname{supp}(j))\cap\Sigma_0=\varnothing$.
	
Then \cite[Thm. 1.2]{GinouxMurro2022} guarantees the existence of a unique solution $(F_E,F_B)\in\Gamma(\V^{n-k}\oplus\V^k)$ to \eqref{Eq: reduced dynamical Maxwell system}.
Moreover \cite[Prop. 3.3]{GinouxMurro2022} entails \eqref{Eq: finite speed of propagation} and thus $F\in\Omega_{sc}^k(\M)$, where $F=\mathrm{d}t\wedge\ast_{h_t}F_E+F_B$.
	
It remains to prove that \eqref{Eq: reduced constraint Maxwell system} holds ---notice that this would also prove that $F\in\Omega_{c,\n}^k(\M)$.
In fact by direct inspection we find
\begin{align*}
	\mathcal{L}_{\partial_t}\mathrm{d}_\Sigma[\beta^{-1}F_E]
	&=\mathrm{d}_\Sigma\mathcal{L}_{\partial_t}[\beta^{-1}F_E]
	\\&=(-1)^{(n-k+1)(k+1)}\cancel{\mathrm{d}_\Sigma^2[\ast_{h_t}\beta F_B]}
	+(-1)^{(n-k)(k+1)}\mathrm{d}_\Sigma[\beta\ast_{h_t} j_B]
	&\textrm{Eq. }\eqref{Eq: E dynamical equation}
	\\&=(-1)^{n-k}\mathcal{L}_{\partial_t}[\beta^{-1}j_E]
	&\textrm{Eq. }\eqref{Eq: j continuity equation}
	\\
	\mathcal{L}_{\partial_t}\mathrm{d}_\Sigma F_B
	&=\mathrm{d}_\Sigma\mathcal{L}_{\partial_t}F_B
	=
	\cancel{\mathrm{d}_\Sigma^2[\ast_{h_t}F_E]}
	+\mathrm{d}_\Sigma[\ast_{h_t}\zeta_E]
	=\mathcal{L}_{\partial_t}\zeta_B
	&\textrm{Eq. }\eqref{Eq: B dynamical equation}
	\\
	\mathcal{L}_{\partial_t}\iota_{\partial\Sigma}^*\beta^{-1}F_E
	&=\iota_{\partial\Sigma}^*\mathcal{L}_{\partial_t}[\beta^{-1}F_E]
	=0
	&\textrm{Eq. }\eqref{Eq: E dynamical equation}
	-\eqref{Eq: B boundary condition}\,,
\end{align*}
where in the last equality we also used that $\n\lrcorner j=0$ is equivalent to $\iota_{\bSigma}^*[\ast_{h_t}j_B]=0$.
The latter equations proves that \eqref{Eq: reduced constraint Maxwell system} is fulfilled once is fulfilled by the initial datum $F_0$: This is the case by assumption.
\end{proof}

\section{Existence of Green operators and pre-symplectic structures}\label{sec:5}
In this section we establish the existence of the Green operators for the differential operator $D=\delta+\mathrm{d}$ acting on $k$-forms and with boundary conditions \eqref{Eq: F normal-boundary condition}.
To this end, we will profit from \cite{BaerGinoux2011,Dappiaggi-Drago-Longhi-2020,GinouxMurro2022}.
For later convenience we recall that $\Omega^k_{sfc}(\M)$ (\textit{resp.} $\Omega^k_{spc}(\M)$) denotes the space of strictly future- (\textit{resp.} past-) compactly supported $k$-forms that is, of all $F\in\Omega^k(\M)$ such that $\operatorname{supp}(F)\subset J^-(K)$ (\textit{resp.} $\operatorname{supp}(F)\subset J^+(K)$) for a suitable compact subset $K\subset\M$.
We also set $\Omega^k_{sc}(\M):=\Omega^k_{sfc}(\M)\cup\Omega^k_{spc}(\M)$.
Similarly $\Omega^k_{fc}(\M)$ (\textit{resp.} $\Omega^k_{pc}(\M)$) denotes the space of future- (\textit{resp.} past-) compactly supported $k$-forms that is, of all $F\in\Omega^k(\M)$ such that $\operatorname{supp}(F)\cap J^+(x)$ (\textit{resp.} $\operatorname{supp}(F)\cap J^-(x)$) is compact for all $x\in\M$.
We set $\Omega^k_{tc}(\M):=\Omega^k_{fc}(\M)\cup\Omega^k_{pc}(\M)$.

\begin{proposition}\label{Prop: advanced, retarded propagators}
	Let $k\in\{0,\ldots,n\}$ and let $D\colon\Omega^k(\M)\to\Omega^{k-1}(\M)\oplus\Omega^{k+1}(\M)$ be the differential operator $D\omega:=\delta\omega+\mathrm{d}\omega$.
	There exists linear operators
	\begin{align*}
		G^+_k\colon\Omega^{k-1}_{c,\n,\delta}(\M)\oplus\Omega^{k+1}_{c,\mathrm{d}}(\M)\to\Omega^k_{spc,\n}(\M)\,,
		\qquad
		G^-_k\colon\Omega^{k-1}_{c,\n,\delta}(\M)\oplus\Omega^{k+1}_{c,\mathrm{d}}(\M)\to\Omega^k_{sfc,\n}(\M)\,,
	\end{align*}
	which fulfil the following properties:
	\begin{align}
		\label{Eq: advanced, retarded propagators - left inverse property with d}
		\mathrm{d}G^\pm_k(\alpha_{k-1}\oplus\zeta_{k+1})
		&=\zeta_{k+1}
		\\
		\label{Eq: advanced, retarded propagators - left inverse property with delta}
		\delta G^\pm_k(\alpha_{k-1}\oplus\zeta_{k+1})
		&=\alpha_{k-1}
		\\
		\label{Eq: advanced, retarded propagators - right inverse property}
		G^\pm_k(\delta\omega_k\oplus\mathrm{d}\omega_k)
		&=\omega_k
		\qquad\forall\omega_k\in\Omega_{c,\n}^k(\M)
		\\
		\label{Eq: advanced, retarded propagators - support property}
		\operatorname{supp}G^\pm_k(\alpha_{k-1}\oplus\zeta_{k+1})
		&\subseteq J^\pm[\operatorname{supp}(\alpha_{k-1})\cup\operatorname{\supp}(\zeta_{k+1})]\,.
	\end{align}
	Moreover, the $G_k^\pm$ can be extended to
	\begin{align}\label{Eq: advanced, retarded propagators - extension of domains}
		G_k^+\colon
		\Omega_{spc,\n,\delta}^{k-1}(\M)
		\oplus\Omega_{spc,\mathrm{d}}^{k+1}(\M)
		\to\Omega_{spc,\n}^k(\M)\,,
		\quad
		G_k^-\colon
		\Omega_{sfc,\n,\delta}^{k-1}(\M)
		\oplus\Omega_{sfc,\mathrm{d}}^{k+1}(\M)
		\to\Omega_{sfc,\n}^k(\M)\,,
	\end{align}	
	still preserving properties \eqref{Eq: advanced, retarded propagators - left inverse property with d}-\eqref{Eq: advanced, retarded propagators - support property}.
	
	Finally, if $G_k:=G_k^+-G_k^-$, then there exists a short exact sequence
	\begin{align}\label{Eq: short exact sequence}
		\{0\}\to\Omega^k_{c,\n}(\M)
		\stackrel{D}{\longrightarrow}
		\Omega^{k-1}_{c,\n,\delta}(\M)
		\oplus\Omega^{k+1}_{c,\mathrm{d}}(\M)
		\stackrel{G_k}{\longrightarrow}
		\Omega^k_{sc,\n}(\M)
		\stackrel{D}{\longrightarrow}
		\delta\Omega^k_{sc,\n}(\M)
		\oplus\mathrm{d}\Omega^k_{sc}(\M)
		\to\{0\}\,.
	\end{align}
\end{proposition}
\begin{proof}
	Let $k\in\{0,\ldots, n\}$.
	Following \cite{BaerGinoux2011,Dappiaggi-Drago-Longhi-2020,Dappiaggi-Lang-2012,Dimock-92,GinouxMurro2022}, we define $G^+_k\colon\Omega^{k-1}_{c,\n,\delta}(\M)\oplus\Omega^{k+1}_{c,\mathrm{d}}(\M)\to\Omega^k_{spc,\n}(\M)$ so that $G_k^+(\alpha_{k-1}\oplus\zeta_{k+1})$ is the unique solution $\omega_k\in\Omega^k(\M)$ to the initial-value problem with boundary conditions
	\begin{align}\label{Eq: advanced, retarded propagators - initial-value problem}
		\mathrm{d}\omega_k=\zeta_{k+1}\,,
		\qquad
		\delta \omega_k=\alpha_{k-1}\,,
		\qquad
		\n\lrcorner \omega_k=0\,,
		\qquad
		\omega_k|_\Sigma=0\,,
	\end{align}
	where $\Sigma$ is an arbitrary but fixed Cauchy surface such that $J^-(\Sigma)\cap[\operatorname{supp}(\alpha_{k-1})\cup\operatorname{\supp}(\zeta_{k+1})]=\varnothing$.
	Existence and uniqueness of $G_k^+(\alpha_{k-1}\oplus\zeta_{k+1})$ follows from Theorem \ref{Thm: well-posedness for Faraday's tensor initial-value problem with boundary conditions}, moreover, $G_k^+$ is easily shown to be linear and independent on the chosen $\Sigma$.
	The map $G^-_k$ is similarly defined by assigning vanishing Cauchy data on a Cauchy surface $\Sigma$ so that $J^+(\Sigma)\cap[\operatorname{supp}(\alpha_{k-1})\cup\operatorname{\supp}(\zeta_{k+1})]=\varnothing$.
	
	Equations \eqref{Eq: advanced, retarded propagators - left inverse property with d}-\eqref{Eq: advanced, retarded propagators - left inverse property with delta} follow from the definition of $G_k^\pm\alpha$ while the inclusion \eqref{Eq: advanced, retarded propagators - support property} is a consequence of \eqref{Eq: finite speed of propagation}.
	Finally, Equation \eqref{Eq: advanced, retarded propagators - right inverse property} follows from the uniqueness of \eqref{Eq: advanced, retarded propagators - initial-value problem} together with the condition $\n\lrcorner\omega_k=0$.
	Notice that the latter condition is necessary for \eqref{Eq: advanced, retarded propagators - right inverse property} as the latter equation implies $\n\lrcorner\omega_k=\n\lrcorner G_k^\pm(\delta\omega_k\oplus\mathrm{d}\omega_k)=0$.
	
	The extension \eqref{Eq: advanced, retarded propagators - extension of domains} is obtained by using property \eqref{Eq: advanced, retarded propagators - support property}, \textit{cf.} \cite[Thm. 3.8]{Baer_2015} whose proof we mimic for the sake of self-containedness of the article.
	To wit, let $\alpha_{k-1}\in\Omega_{spc,\n,\delta}^{k-1}(\M)$ and $\zeta_{k+1}\in\Omega_{spc,\mathrm{d}}^{k+1}(\M)$.
	We define $G_k^+(\alpha_{k-1}\oplus\zeta_{k+1})$ as follows ---a similar argument goes for $G_k^-$.
For fixed $x\in\M$, let $K_x:=J^-(x)\cap[\operatorname{\supp}(\alpha_{k-1}\oplus\zeta_{k+1})]$.
	Then $K_x$ is compact and we may choose $\chi\in C^\infty_c(\M)$ such that $\chi|_{K_x}=1$.
For any such $\chi$ we set
	\begin{align}\label{Eq: advanced, retarded propagator - extension of domain definition}
		G_k^+(\alpha_{k-1}\oplus\zeta_{k+1})|_x
		:=G_k^+(\chi\alpha_{k-1}\oplus\chi\zeta_{k+1})|_x\,.
	\end{align}
Note that $\operatorname{\supp}(\chi)$ being compact ensures that $\chi\alpha_{k-1}$ and $\chi\zeta_{k+1}$ are compactly supported.
	Moreover, $\operatorname{\supp}(\mathrm{d}[\chi\zeta_{k+1}])\cap J^-(x)=\varnothing$ and similarly $\operatorname{supp}(\delta[\chi\alpha_{k-1}])\cap J^-(x)=\varnothing$.
	On account of property \eqref{Eq: advanced, retarded propagators - support property} this entails that $G_k^+(\chi\alpha_{k-1}\oplus\chi\zeta_{k+1})|_x$ is well-posed and defines the wanted extension.
	
	The resulting map $G_k^+$ is independent on the particular choice of $\chi$.
	Indeed, any pair of functions $\chi,\chi'$ with the above properties fulfil $\operatorname{supp}[(\chi-\chi')(\alpha_{k-1}\oplus\zeta_{k+1})]\cap J^-(x)=\varnothing$, therefore, $G_k^+[\chi\alpha_{k-1}\oplus\chi\zeta_{k+1}]|_x=G_k^+[\chi'\alpha_{k-1}\oplus\chi'\zeta_{k+1}]|_x$.
	
	The $\chi$-independence implies linearity of the resulting map $G_k^+$.
	Indeed, if $\alpha_{k-1}\oplus\zeta_{k+1}$, $\alpha_{k-1}'\oplus\zeta_{k+1}'$ are in $\Omega_{spc,\n,\delta}^{k-1}(\M)\oplus\Omega_{spc,\mathrm{d}}^{k+1}(\M)$, then for all $x\in\M$ we may choose $\chi\in C^\infty(\M)$ so that $\chi=1$ on $J^-(x)\cap[\operatorname{supp}(\alpha_{k-1}\oplus\zeta_{k+1})\cup\operatorname{supp}(\alpha_{k-1}'\oplus\zeta_{k+1}')]$, thus
	\begin{align*}
		G_k^+[(\alpha_{k-1}+\alpha_{k-1}')\oplus(\zeta_{k+1}+\zeta_{k+1}')]|_x
		&=G_k^+[(\chi\alpha_{k-1}+\chi\alpha_{k-1}')\oplus(\chi\zeta_{k+1}+\chi\zeta_{k+1}')]|_x
		\\
		&=G_k^+[\chi\alpha_{k-1}\oplus\chi\zeta_{k+1}]|_x
		+G_k^+[\chi\alpha_{k-1}'\oplus\chi\zeta_{k+1}']|_x
		\\
		&=G_k^+[\alpha_{k-1}\oplus\zeta_{k+1}]|_x
		+G_k^+[\alpha_{k-1}'\oplus\zeta_{k+1}']|_x\,.
	\end{align*}
	Property \eqref{Eq: advanced, retarded propagators - support property} follows from Equation \eqref{Eq: advanced, retarded propagator - extension of domain definition}.
The same holds for properties \eqref{Eq: advanced, retarded propagators - left inverse property with delta}-\eqref{Eq: advanced, retarded propagators - left inverse property with d}.
Note also that, because it is of vanishing order, the boundary condition $\n\lrcorner G_k^+(\alpha_{k-1}\oplus\zeta_{k+1})=0$ is also a straightforward consequence of the definition of $G_k^+$.
For what concerns \eqref{Eq: advanced, retarded propagators - right inverse property} we observe that, for all $\omega_k\in\Omega^k_{spc,\n}(\M)$ it holds
	\begin{align*}
		G_k^+(\delta\omega_k\oplus\mathrm{d}\omega_k)|_x
		=G_k^+(\chi\delta\omega_k\oplus\chi\mathrm{d}\omega_k)|_x
		=G_k^+(\delta\chi\omega_k\oplus\mathrm{d}\chi\omega_k)|_x
		=\chi\omega_k|_x
		=\omega_k|_x\,,
	\end{align*}
	where we used $\operatorname{supp}(\mathrm{d}\chi)\cap\operatorname{supp}(\alpha_{k-1}\oplus\zeta_{k+1})\cap J^-(x)=\varnothing$.
	
	We now prove the exactness of \eqref{Eq: short exact sequence}.
	To begin with, notice that if $\alpha_k\in\Omega_{c,\n}^k(\M)$ is such that $D\alpha_k=0$ ---\textit{i.e.} $\delta\alpha_k=0$ and $\mathrm{d}\alpha_k=0$--- then we have $\alpha_k=G_k^+(\delta\alpha_k,\mathrm{d}\alpha_k)=0$: This shows exactness in the first arrow of \eqref{Eq: short exact sequence}.
	
	If $\alpha_k\in\Omega^k_{c,\n}(\M)$ then $G_kD\alpha_k=G_k^+(\delta\alpha_k,\mathrm{d}\alpha_k)-G_k^-(\delta\alpha_k,\mathrm{d}\alpha_k)=\alpha_k-\alpha_k=0$, proving that $D\Omega_{c,\n}^k(\M)\subset\ker G_k$.
	Conversely, if $\alpha_{k-1}\oplus\zeta_{k+1}\in\Omega_{c,\n,\delta}^{k-1}(\M)\oplus\Omega^k_{c,\mathrm{d}}(\M)$ is such that $G_k(\alpha_{k-1}\oplus\zeta_{k+1})=0$ then $G_k^+(\alpha_{k-1}\oplus\zeta_{k+1})=G_k^-(\alpha_{k-1}\oplus\zeta_{k+1})\in\Omega_{c,\n}^k(\M)$ is such that
	\begin{align*}
		DG_k^+(\alpha_{k-1}\oplus\zeta_{k+1})
		=\delta G_k^+(\alpha_{k-1}\oplus\zeta_{k+1})
		+\mathrm{d}G_k^+(\alpha_{k-1}\oplus\zeta_{k+1})
		=\alpha_{k-1}\oplus\zeta_{k+1}\,.
	\end{align*}
	This proves exactness of \eqref{Eq: short exact sequence} in the second arrow.
	
	Let $\alpha_{k-1}\oplus\zeta_{k+1}\in\Omega_{c,\n,\delta}^{k-1}(\M)\oplus\Omega^{k+1}_{c,\mathrm{d}}(\M)$: Then	$\delta G_k(\alpha_{k-1}\oplus\zeta_{k+1})=\delta G_k^+(\alpha_{k-1}\oplus\zeta_{k+1})-\delta G_k^-(\alpha_{k-1}\oplus\zeta_{k+1})=\alpha_{k-1}-\alpha_{k-1}=0$, and similarly $\mathrm{d}G_k(\alpha_{k-1}\oplus\zeta_{k+1})=0$.
	This shows that $DG_k(\alpha_{k-1}\oplus\zeta_{k+1})=0$ and thus $G_k[\Omega_{c,\n,\delta}^{k-1}(\M)\oplus\Omega^{k+1}_{c,\mathrm{d}}(\M)]\subset\ker D$.
	Moreover, let $\omega_k\in\Omega_{sc,\n}^k(\M)$ be such that $D\omega_k=0$.
	Consider a function $\chi\in C^\infty(\M)$ such that $\mathrm{d}\chi\in\operatorname{span}\mathrm{d}t$ and such that $\chi(t)=1$ for $t\geq t_0$, $t_0\in\mathbb{R}$ being arbitrary, and $\chi(t)=0$ for $t\leq -t_0$.
	Let $\omega_k^+:=\chi\omega_k$ and $\omega_k^-:=(1-\chi)\omega_k$.
	Then $\omega_k^+\in\Omega_{spc,\n}^k(\M)$ and $\omega_k^-\in\Omega_{sfc,\n}^k(\M)$.
	Moreover, $\delta\omega_k^+=-\delta\omega_k^-\in\Omega_{c,\n}^{k-1}(\M)$ and similarly $\mathrm{d}\omega_k^\pm\in\Omega_c^{k+1}(\M)$.
	Finally
	\begin{align*}
		G_k(\delta\omega_k^+\oplus\mathrm{d}\omega_k^+)
		&=G_k^+(\delta\omega_k^+\oplus\mathrm{d}\omega_k^+)
		-G_k^-(\delta\omega_k^+\oplus\mathrm{d}\omega_k^+)
		\\
		&=G_k^+(\delta\omega_k^+\oplus\mathrm{d}\omega_k^+)
		+G_k^-(\delta\omega_k^-\oplus\mathrm{d}\omega_k^-)\\
		&=\omega_k^+
		+\omega_k^-\\
		&=\omega_k\,,
	\end{align*}
	where we used the extension \eqref{Eq: advanced, retarded propagators - extension of domains}.
	This shows exactness in the third arrow of \eqref{Eq: short exact sequence}.
	
	Finally, let $\alpha_k\in\Omega_{sc,\n}^k(\M)$ and $\beta_k\in\Omega_{sc}^k(\M)$.
	We wish to prove the existence of $\omega_k\in\Omega_{sc,\n}^k(\M)$ such that $D\omega_k=\delta\alpha_k\oplus\mathrm{d}\beta_k$, that is, $\delta\omega_k=\delta\alpha_k$ and $\mathrm{d}\omega_k=\mathrm{d}\beta_k$.
	To this avail, we consider $\chi\in C^\infty(\M)$ as above and let $\alpha_k=\alpha^+_k+\alpha^-_k$, where $\alpha^+_k:=\chi\alpha_{k-1}$ and $\alpha^-_k:=(1-\chi)\alpha^-_k$ and similarly $\beta_k=\beta_k^++\beta_k^-$.
	Notice that, per construction $\alpha_k^+\in\Omega_{spc,\n}^+(\M)$, $\alpha_k^-\in\Omega_{sfc,\n}^k(\M)$ and similarly  $\beta_k^+\in\Omega_{spc}^k(\M)$ and $\beta_k^-\in\Omega_{sfc}^k(\M)$.
	We then set $\omega_k:=G_k^+(\delta\alpha_k^+\oplus\mathrm{d}\beta_k^+)+G^-_k(\delta\alpha_k^-\oplus\mathrm{d}\beta_k^-)$.
	Per definition $\omega_k\in\Omega_{sc,\n}^k(\M)$, moreover, $D\omega_k=\delta\alpha_k^+\oplus\mathrm{d}\beta_k^++\delta\alpha_k^-\oplus\mathrm{d}\beta_k^-=\delta\alpha_k\oplus\mathrm{d}\beta_k$, where we used the extension \eqref{Eq: advanced, retarded propagators - extension of domains}.
	This shows exactness of \eqref{Eq: short exact sequence} in the fourth and last arrow.
\end{proof}

\begin{remark}
	From \eqref{Eq: advanced, retarded propagators - extension of domains} it follows that the causal propagator $G_k$ extends to a linear map $G_k\colon\Omega_{tc,\n,\delta}^{k-1}(\M)\oplus\Omega_{tc,\mathrm{d}}^{k+1}(\M)\to\Omega_{\n}^k(\M)$, \textit{cf.} \cite[Thm. 3.8]{Baer_2015}.
	Furthermore, one may generalize the exact sequence \eqref{Eq: short exact sequence} by relaxing the compactness support assumption to timelike compactness, while dropping the spacelike compactness condition:
	\begin{align}\label{Eq: generalized short exact sequence}
		\{0\}\to\Omega^k_{tc,\n}(\M)
		\stackrel{D}{\longrightarrow}
		\Omega^{k-1}_{tc,\n,\delta}(\M)
		\oplus\Omega^{k+1}_{tc,\mathrm{d}}(\M)
		\stackrel{G_k}{\longrightarrow}
		\Omega^k_{\n}(\M)
		\stackrel{D}{\longrightarrow}
		\delta\Omega^k_{\n}(\M)
		\oplus\mathrm{d}\Omega^k(\M)
		\to\{0\}\,.
	\end{align}
\end{remark}

The exactness of \eqref{Eq: short exact sequence} leads to the following isomorphism, which provides a complete description of the solution space to Maxwell's equations by generalizing the well-known situation on a globally hyperbolic spacetime without boundary:
\begin{align}\label{Eq: solution space isomorphism}
	\operatorname{Sol}_{sc,\n}^k(\M)
	&:=\{F_k\in\Omega_{sc}^k(\M)\,|\,\delta F_k=0\,,\,\mathrm{d}F_k=0\,,\,\n\lrcorner F_k=0\}
	\\
	&\simeq G_k[\Omega_{c,\n,\delta}^{k-1}(\M)\oplus \Omega_{c,\mathrm{d}}^{k+1}(\M)]
	\simeq\frac{\Omega_{c,\n,\delta}^{k-1}(\M)\oplus\Omega_{c,\mathrm{d}}^{k+1}(\M)}{D\Omega_{c,\n}^k(\M)}\,.
\end{align}

\subsection{Causal propagator and the pre-symplectic structure}
We conclude the paper by endowing the space of homogeneous solutions to the Faraday Cauchy problem with a pre-symplectic form.
The latter is constructed out of the causal propagators $\{G_k\}_{k=1}^n$ introduced in Proposition \ref{Prop: advanced, retarded propagators}.
The resulting pre-symplectic structure requires to consider all $k$-forms at once in a non-trivial fashion.
To this avail we set $\Omega^\oplus(\M):=\oplus_{k=0}^n\Omega^k(\M)$: An element of this latter space will be denoted by $\underline{F}=\sum_{k=0}^nF_k$, $F_k\in\Omega^k(\M)$.
The natural pairing $\Omega^\oplus(\M)^2\to\mathbb{R}$ inherited from the pairings on $\Omega^k(\M)$ is denoted by $(\cdot\,,\cdot)_\oplus$.
Let
\begin{align}\label{Eq: sequences of solution forms}
	\mathcal{S}&:=\{
	\underline{F}\in\Omega^\oplus_{sc,\n}(\M)\,|\,D\underline{F}=0\}
	\\
	&=\{
	\underline{F}\in\Omega^\oplus_{sc}(\M)\,|\,F_k\in\Omega^k_{sc,\n}(\M)\,,\,
	\mathrm{d}F_k=0\,,\,
	\delta F_k=0\,,\forall k\in\{0,\ldots,n\}
	\}\,.
\end{align}
Notice that $F_0=0$, moreover,
\begin{align}\label{Eq: symmetry of D on forms with boundary conditions}
	(D\underline{F}^{(1)},\underline{F}^{(2)})_\oplus
	=(\underline{F}^{(1)},D\underline{F}^{(2)})_\oplus
	\quad
	\forall\underline{F}^{(1)},\underline{F}^{(2)}\in\Omega^\oplus_{\n}(\M)\,,
	\operatorname{supp}(\underline{F}^{(1)})\cap\operatorname{supp}(\underline{F}^{(2)})\textrm{ compact}\,.
\end{align}

A direct application of Proposition \ref{Prop: advanced, retarded propagators} leads to the following isomorphism of vector spaces:
\begin{align}\label{Eq: isormorphism for sequences of solutions forms}
	\bigoplus_{k=1}^n\frac{\Omega_{c,\n,\delta}^{k-1}(\M)\oplus\Omega_{c,\mathrm{d}}^{k+1}(\M)}{D\Omega_{c,\n}^k(\M)}
	\simeq\bigoplus_{k=1}^n\operatorname{Sol}^k_{sc,\n}(\M)
	=\mathcal{S}
	\qquad
	\underline{\alpha}\oplus\underline{\zeta}\mapsto\underline{G}\,(\underline{\alpha}\oplus\underline{\zeta})\,,
\end{align}
where $\underline{G}:=\oplus_{k=1}^nG_k$.

\begin{proposition}\label{Prop: pre-symplectic structure on spacelike solutions}
With the notation introduced above, let $\sigma_{\mathcal{S}}\colon\mathcal{S}\times \mathcal{S}\to\mathbb{R}$ be defined by
	\begin{align}\label{Eq: pre-symplectic structure on spacelike solutions}
		\sigma_{\mathcal{S}}(\underline{F}^{(1)},\underline{F}^{(2)})
		:=(D\underline{F}^{(1),+},\underline{F}^{(2)})_\oplus\,,
	\end{align}
	where $\underline{F}^{(1)}=\underline{F}^{(1),+}+\underline{F}^{(1),-}$, $\underline{F}^{(1),+}\in\Omega^\oplus_{sfc,\n}(\M)$, $\underline{F}^{(1),-}\in\Omega^\oplus_{spc,\n}(\M)$, is an arbitrary decomposition of $\underline{F}^{(1)}$ in strictly future/past compactly supported forms.
	
	Then $\sigma_{\mathcal{S}}$ is a well-defined pre-symplectic structure on $\mathcal{S}$.
	Moreover, if $\M$ admits a finite good cover \cite{Bott_1982,Schwarz_1995} it holds
	\begin{align}\label{Eq: degeneracy of pre-symplectic structure on spacelike solutions}
		\sigma_{\mathcal{S}}(\cdot,\underline{F})=0
		\quad\Longleftrightarrow\quad
		\underline{F}=\mathrm{d}\underline{A}=\delta\underline{B}\,,
	\end{align}
	where $\underline{A}\in\Omega^\oplus_{sc}(\M)$ and $\underline{B}\in\Omega^\oplus_{sc,\n}(\M)$ ---in particular $\underline{A}\in\Omega^\oplus_{sc}(\M)$ is such that $\delta\mathrm{d}\underline{A}=0$ and $\n\lrcorner\mathrm{d}\underline{A}=0$.
\end{proposition}

\begin{proof}
	We adapt the arguments of \cite{Benini_2016,Dappiaggi-Drago-Longhi-2020} to the current case.
	To begin with, we observe that a decomposition of the form $\underline{F}=\underline{F}^++\underline{F}^-$ can always be realized by multiplying $\underline{F}$ by a suitable time-dependent function $\chi\in C^\infty(\M)$: Notice that this also preserves the boundary conditions.
	Moreover, if $D\underline{F}=0$ then $D\underline{F}^+=-D\underline{F}^-$, therefore, $D\underline{F}^+\in\Omega^\oplus_c(\M)$.
	This implies that the pairing $(D\underline{F}^{(1),+},\underline{F}^{(2)})_\oplus$ is well-defined for all $\underline{F}^{(1)},\underline{F}^{(2)}\in\mathcal{S}$.
	
	Next we observe that $\underline{F}^{(1)},\underline{F}^{(2)}\mapsto(D\underline{F}^{(1),+},\underline{F}^{(2)})_\oplus$ is in fact independent on the splitting $\underline{F}^{(1)}=\underline{F}^{(1),+}+\underline{F}^{(1),-}$.
	Indeed, if $\underline{F}^{(1)}=\underline{F}^{(1),+\prime}+\underline{F}^{(1),-\prime}$ is another such splitting we have $\underline{F}^{(1),+\prime}-\underline{F}^{(1),+}=\underline{F}^{(1),-}-\underline{F}^{(1),-\prime}$ which ensures that $\underline{F}^{(1),+\prime}-\underline{F}^{(1),+}\in\Omega^k_{c,\n}(\M)$.
	This implies that
	\begin{align*}
		(D\underline{F}^{(1),+\prime},\underline{F}^{(2)})_\oplus
		-(D\underline{F}^{(1),+},\underline{F}^{(2)})_\oplus
		&=(D(\underline{F}^{(1),+\prime}-\underline{F}^{(1),+}),\underline{F}^{(2)})_\oplus
		\\
		&=(\underline{F}^{(1),+\prime}-\underline{F}^{(1),+},D\underline{F}^{(2)})_\oplus
		=0\,,
	\end{align*}
	where we applied Equation \eqref{Eq: symmetry of D on forms with boundary conditions}.
	
	Thus, the map $\sigma_{\mathcal{S}}\colon\mathcal{S}^2\to\mathbb{R}$ is well-defined and readily bilinear.
	We now prove that it is skew-symmetric, therefore, it provides a pre-symplectic structure on $\mathcal{S}$.
	To this avail let $\underline{F}^{(1)},\underline{F}^{(2)}\in\mathcal{S}$ and consider two decompositions $\underline{F}^{(j)}=\underline{F}^{(j),+}+\underline{F}^{(j),-}$, $j\in\{1,2\}$, as above.
	Then repeatedly using Equation \eqref{Eq: symmetry of D on forms with boundary conditions} we have
	\begin{align*}
		\sigma_{\mathcal{S}}(\underline{F}^{(1)},\underline{F}^{(2)})
		&=(D\underline{F}^{(1),+},\underline{F}^{(2)})_\oplus
		\\
		&=(D\underline{F}^{(1),+},\underline{F}^{(2),+})_\oplus
		+(D\underline{F}^{(1),+},\underline{F}^{(2),-})_\oplus
		\\
		&=-(D\underline{F}^{(1),-},\underline{F}^{(2),+})_\oplus
		+(\underline{F}^{(1),+},D\underline{F}^{(2),-})_\oplus
		\\
		&=-(\underline{F}^{(1),-},D\underline{F}^{(2),+})_\oplus
		-(\underline{F}^{(1),+},D\underline{F}^{(2),+})_\oplus
		\\
		&=-(\underline{F}^{(1)},D\underline{F}^{(2),+})_\oplus
		=-\sigma_{\mathcal{S}}(\underline{F}^{(1)},\underline{F}^{(2)})\,.
	\end{align*}

	Finally, let assume that $\M$ has a finite cover and let $\underline{F}\in\mathcal{S}$ be such that $\sigma_{\mathcal{S}}(\underline{F}',\underline{F})=0$.
	We observe that each component $F_k$ of $\underline{F}\in\mathcal{S}$ induces an element, still denoted by $F_k$, of the dual space $H_{k,c,\n}(\M)^*$ where
	\begin{align*}
		H_{k,c,\n}(\M):=\frac{\{\alpha_k\in\Omega^k_{c,\n}(\M)\,|\,\delta\alpha_k=0\}}
		{\delta\Omega^{k+1}_{c,\n}(\M)}\,.
	\end{align*}
	Indeed $F_k([\alpha_k]):=(\alpha_k,F_k)$ is well-defined for all $[\alpha]\in H_{k,c,\n}(\M)$ on account of the identity
	$(\delta\beta_{k+1},F_k)=(\beta_{k+1},\mathrm{d}F_k)=0$ for all $\beta_{k+1}\in\Omega^{k+1}_{c,\n}(\M)$.
	Notice that, since $\M$ has a good cover, $H_{k,c,\n}(\M)^*\simeq H^k(\M)$, where $H^k(\M)$ is the standard $k$-th de Rham cohomology group, \textit{cf.} \cite{Bott_1982, Schwarz_1995} and \cite[App. C]{Dappiaggi-Drago-Longhi-2020}.
	A similar argument shows that the assignment $\alpha_k\mapsto F_k([\zeta_k]):=(\zeta_k,F_k)$ defines an element in $H^k_c(\M)^*\simeq H_{k,\n}(\M)$.
	
	On account of \eqref{Eq: isormorphism for sequences of solutions forms} we have
	\begin{align*}
		\underline{F}'=\underline{G}(\underline{\alpha}\oplus\underline{\zeta})\,,
		\qquad \underline{\alpha}\oplus\underline{\zeta}\in\bigoplus\limits_{k=1}^n
		\frac{\Omega^{k-1}_{c,\n,\delta}(\M)\oplus\Omega^{k+1}_{c,\mathrm{d}}(\M)}{D\Omega^k_{c,\n}(\M)}\,.
	\end{align*}	
	Thus, we may set $\underline{F}^{\prime,+}:=\underline{G}^+(\underline{\alpha}\oplus\underline{\zeta})$ which leads to
	\begin{align*}
		\sigma_{\mathcal{S}}(\underline{F}',\underline{F})
		=(D\underline{G}^+(\underline{\alpha}\oplus\underline{\zeta}),\underline{F})_\oplus
		=(\underline{\alpha}\oplus\underline{\zeta},\underline{F})_\oplus\,.
	\end{align*}
	The condition $\sigma_{\mathcal{S}}(\underline{F}',\underline{F})=0$ and the arbitrariness of $\underline{\alpha}$ implies in particular that $(\alpha_k,F_k)=0$ for all $\alpha_k\in\Omega^k_{c,\n,\delta}(\M)$ and $k\in\{0,\ldots,n\}$.
	This entails that $F_k=0\in H_{k,c,\n}(\M)^*\simeq H^k(\M)$, that is, $F_k=\mathrm{d}A_{k-1}$: Thus, $\underline{F}=\mathrm{d}\underline{A}$.
	With a similar argument, the arbitrariness of $\underline{\zeta}$ leads to $(\zeta_k,F_k)=0$ for all $\zeta_k\in\Omega^k_{c,\mathrm{d}}(\M)$ which implies $F_k=0\in H^k_c(\M)^*\simeq H_{k,\n}(\M)$, therefore $F_k=\delta B_{k+1}$ for $B_{k+1}\in\Omega^{k+1}_{\n}(\M)$.
	
	Conversely, by direct inspection any element $\underline{F}\in\mathcal{S}$ such that $\underline{F}=\mathrm{d}\underline{A}=\delta\underline{B}$ for $\underline{B}\in\Omega^\oplus_{\n}(\M)$ fulfils $\sigma_{\mathcal{S}}(\;,\underline{F})=0$.
\end{proof}

\begin{remarks}
	\noindent
	\begin{enumerate}
		\item
		The pre-symplectic form $\sigma_{\mathcal{S}}$ involves forms of different degrees in a non-trivial fashion.
		In particular, this spoils the possibility of inducing a pre-symplectic form on a single component of $\underline{F}\in\mathcal{S}$.
		At its core, this difficulty is due to the different degrees in the domain and codomain of the operators $G^\pm_k$, \textit{cf.} Proposition \ref{Prop: advanced, retarded propagators}.
		Moreover, the degeneracy space of $\sigma_{\mathcal{S}}$ coincides with the space of spacelike solutions to Maxwell's equation for the electromagnetic potential \cite[Def. 28]{Dappiaggi-Drago-Longhi-2020}.
		These two facts do not allow a clear physical interpretation of the resulting structure.
		
		For the purpose of quantizing the solution space $\operatorname{Sol}^k_{sc,\n}(\M)$ for fixed $k$ it is likely more appropriate to proceed as in \cite{Dappiaggi-Lang-2012}, which is based on the connection with the solution space to the wave operator $\square$.
		For the case at hand, such connection would require the identification of appropriate boundary conditions which guarantee formal self-adjointness of $\square$.
		The latter can be easily determined by observing that any $F\in\operatorname{Sol}^k_{sc,\n}(\M)$ fulfils $\n\lrcorner F=0$ as well as $\n\lrcorner\mathrm{d}F=0$.
		Moreover, $(\square\alpha_k,\beta_k)=(\alpha_k,\square\beta_k)$ if $\alpha_k,\beta_k\in\Omega^k(\M)$ are such that $\operatorname{supp}(\alpha_k)\cap\operatorname{supp}(\beta_k)$ is compact and $\n\lrcorner\alpha_k=\n\lrcorner\beta_k=0$ as well as $\n\lrcorner\mathrm{d}\alpha_k=\n\lrcorner\mathrm{d}\beta_k=0$.
		Forms abiding by these boundary conditions were investigated in \cite{Dappiaggi-Drago-Longhi-2020}, which deals with the quantization of the electromagnetic vector potential in the framework of gauge theories.
		\item
		Similarly to \cite{Benini_2016,Dappiaggi-Drago-Longhi-2020} one may promote the isomorphism of vector spaces \eqref{Eq: isormorphism for sequences of solutions forms} to an isomorphism of pre-symplectic vector spaces.
		This requires to define a pre-symplectic form
		\begin{multline*}
			\varsigma_{\mathcal{S}}\colon \left(\bigoplus_{k=1}^n
			\frac{\Omega^{k-1}_{c,\n,\delta}(\M)\oplus\Omega^{k+1}_{c,\mathrm{d}}(\M)}{D\Omega^k_{c,\n}(\M)}\right)^2
			\to\mathbb{R}
			\\
			\varsigma_{\mathcal{S}}(\underline{\alpha}^{(1)}\oplus\underline{\zeta}^{(1)},\underline{\alpha}^{(2)}\oplus\underline{\zeta}^{(2)})
			:=(\underline{\alpha}^{(1)}\oplus\underline{\zeta}^{(1)},\underline{G}(\underline{\alpha}^{(2)}\oplus\underline{\zeta}^{(2)}))_\oplus\,,
		\end{multline*}
		from which $\sigma_{\mathcal{S}}(\underline{G}(\underline{\alpha}^{(1)}\oplus\underline{\zeta}^{(1)}),\underline{G}(\underline{\alpha}^{(2)}\oplus\underline{\zeta}^{(2)}))=\varsigma_{\mathcal{S}}(\underline{\alpha}^{(1)}\oplus\underline{\zeta}^{(1)},\underline{\alpha}^{(2)}\oplus\underline{\zeta}^{(2)})$ follows by decomposing $\underline{G}(\underline{\alpha}^{(1)}\oplus\underline{\zeta}^{(1)})=\underline{G}^+(\underline{\alpha}^{(1)}\oplus\underline{\zeta}^{(1)})-\underline{G}^-(\underline{\alpha}^{(1)}\oplus\underline{\zeta}^{(1)})$ together with the observation that $D\underline{G}^+(\underline{\alpha}^{(1)}\oplus\underline{\zeta}^{(1)})=\underline{\alpha}^{(1)}\oplus\underline{\zeta}^{(1)}$.
	\end{enumerate}
\end{remarks}


\vspace{0.5cm}
\begin{thebibliography}{99}

\bibitem{Ake-Flores-Sanchez-18} 
L.~Ak\'e, J.~L.~Flores and M.~S\'anchez,
\textit{Structure of globally hyperbolic spacetimes
with timelike boundary.}  Rev. Mat. Iberoam. {37} (2021), 45--94.

\bibitem{Baer_2015}
C.~B\"ar, 
\textit{Green-hyperbolic operators on globally hyperbolic spacetimes,} Comm.\ 
Math.\ Phys.\  {\bf 333} (2015), 1585.


\bibitem{BaerGinoux2011}
C.~B\"ar and N.~Ginoux,
\textit{Classical and Quantum Fields on Lorentzian Manifolds},
Global Differential Geometry, Springer Proceedings in Math. 17, Springer-Verlag (2011) 359--400.


\bibitem{wave}
C.~B\"ar, N.~Ginoux and F.~Pf\"affle,
\textit{Wave equations on Lorentzian Manifolds and Quantization.}
ESI Lectures in Mathematics and Physics (2007).


\bibitem{Benini_2016}
M.~Benini,
\textit{Optimal space of linear classical observables for Maxwell k-forms via spacelike and timelike compact de Rham cohomologies},
J. Math. Phys. {\bf 57} (2016), 053502.

\bibitem{Benini_Dappiaggi_Hack_Schenkel_2014}
M.~Benini, C.~Dappiaggi, T.-P.~Hack and A.~Schenkel,
\textit{A $C^*$-algebra for quantized principal $U(1)$-connections on globally hyperbolic Lorentzian manifolds},
Commun. Math. Phys. 332, 477 (2014)

\bibitem{Benini_Dappiaggi_Schenkel_2014}
M.~Benini, C.~Dappiaggi and A.~Schenkel,
\textit{Quantized Abelian principal connections on Lorentzian manifolds},
Commun. Math. Phys. 330, 123 (2014).

\bibitem{Benini_Dappiaggi_Schenkel_2018}
M.~Benini, C.~Dappiaggi and A.~Schenkel,
\textit{Algebraic quantum field theory on spacetimes with timelike boundary},
Annales Henri Poincar\'e, 19(8), 2401-2433 (2018).

\bibitem{Bott_1982}
R.~Bott and L.W.~Tu,
\textit{Differential Forms in Algebraic Topology}, Springer, Berlin (1982).

 \bibitem{aqft2} 
R.~Brunetti, C.~Dappiaggi, K.~Fredenhagen and J.~Yngvason, \textit{Advances in 
algebraic quantum field theory.} Springer (2015)

\bibitem{Dappiaggi-Drago-Longhi-2020}
C.~Dappiaggi, N.~Drago and R.~Longhi,
\textit{On Maxwell's equations on globally hyperbolic spacetimes with timelike boundary},
Ann. Henri Poincar\'e, {\bf 21} (2020), 2367--2409.

\bibitem{Dappiaggi-Lang-2012}
C.~Dappiaggi and B.~Lang, \textit{Quantization of Maxwell's equations on curved 
backgrounds and general local covariance.} Lett. Math. Phys. {\bf 101} (2012), 
265--287.

\bibitem{Dimock-92}
J.~Dimock,
\textit{Quantized Electromagnetic field on a manifold},
Rev. Math. Phys. {\bf 04} (1992), 223--233.

\bibitem{DragoGinouxMurro2022} N.~Drago, N.~Ginoux and S.~Murro, \textit{M\o
ller operators and Hadamard states for Dirac
fields with MIT boundary conditions,} Doc. Math. \textbf{27} (2022), 1643--1687.

\bibitem{FK} 
K.~Fredenhagen and K.~Rejzner,
  \textit{Quantum field theory on curved spacetimes: Axiomatic framework and 
examples,}
  J.\ Math.\ Phys.\  {\bf 57} (2016), 031101 .

\bibitem{gerard} 
C.~G\'erard, \textit{Microlocal Analysis of Quantum Fields on Curved 
Spacetimes.} ESI Lectures in Mathematics and Physics (2019). 

\bibitem{GinouxMurro2022} N.~Ginoux and S.~Murro, \textit{On the Cauchy 
problem for Friedrichs systems on globally hyperbolic manifolds with timelike 
boundary}, Ad. Differential Equations {\bf 27} (2022), 497--542.

\bibitem{Sak}
J.J.~Sakurai, \textit{Modern Quantum Mechanics.} Addison-Wesley Publishing 
(1994).

\bibitem{Schwarz_1995}
G. Schwarz,
\textit{Hodge Decomposition - A Method for Solving Boundary Value Problems}, Springer, Berlin (1995).
\end{thebibliography}
\end{document}